\documentclass[12pt]{amsart}

\usepackage{amsmath}
\usepackage{a4wide}
\usepackage[utf8]{inputenc}
\usepackage{amssymb}
\usepackage{amsopn}
\usepackage{epsfig}
\usepackage{amsfonts}
\usepackage{latexsym}
\usepackage{amsthm}
\usepackage{enumerate}
\usepackage[UKenglish]{babel}
\usepackage{verbatim}
\usepackage{color}

\newtheorem{theorem}{Theorem}[section]
\newtheorem{lemma}[theorem]{Lemma}
\newtheorem{proposition}[theorem]{Proposition}

\theoremstyle{definition}

\newtheorem{example}[theorem]{Example}

\theoremstyle{remark}

\numberwithin{equation}{section}

\newcommand{\R}{\ensuremath{\mathbb{R}}}

\renewcommand{\u}{\mathcal{U}}

\newcommand{\set}[1]{\left\{#1\right\}}
\newcommand{\U}{\widetilde{\mathcal{U}}}

\newcommand{\f}{\infty}

\newcommand{\G}{\Gamma_\alpha}

\begin{document}

\title{Unique expansions and intersections of Cantor sets}

\author{Simon Baker}
\address{Department of Mathematics and Statistics, Whiteknights, Reading, RG6 6AX, UK}
\email{simonbaker412@gmail.com}

\author{Derong Kong}
\address{School of Mathematical Science, Yangzhou University, Yangzhou, Jiangsu 225002, People's Republic of China}
\email{derongkong@126.com}

\date{\today}



\begin{abstract}
To each $\alpha\in(1/3,1/2)$ we associate the Cantor set $$\Gamma_{\alpha}:=\Big\{\sum_{i=1}^{\infty}\epsilon_{i}\alpha^i: \epsilon_i\in\{0,1\},\,i\geq 1\Big\}.$$ In this paper we consider the intersection $\Gamma_\alpha \cap (\Gamma_\alpha + t)$ for any translation $t\in\R$. We pay special attention to those $t$ with a unique $\{-1,0,1\}$ $\alpha$-expansion, and study the set
 $$D_\alpha:=\{\dim_H(\Gamma_\alpha \cap (\Gamma_\alpha + t)):t \textrm{ has a unique }\{-1,0,1\}\,\alpha\textrm{-expansion}\}.$$We prove that there exists a transcendental number $\alpha_{KL}\approx 0.39433\ldots$ such that: $D_\alpha$  is finite for $\alpha\in(\alpha_{KL},1/2),$ $D_{\alpha_{KL}}$ is infinitely countable, and $D_{\alpha}$ contains an interval for $\alpha\in(1/3,\alpha_{KL}).$ We also prove that $D_\alpha$ equals $[0,\frac{\log 2}{-\log \alpha}]$ if and only if $\alpha\in (1/3,\frac {3-\sqrt{5}}{2}].$

As a consequence of our investigation we   prove some results on the possible values of $\dim_{H}(\Gamma_\alpha \cap (\Gamma_\alpha + t))$ when $\Gamma_\alpha \cap (\Gamma_\alpha + t)$ is a self-similar set. We also give examples of $t$ with a continuum of $\{-1,0,1\}$ $\alpha$-expansions for which we can explicitly calculate $\dim_{H}(\G\cap(\G+t)),$ and for which $\G\cap (\G+t)$ is a self-similar set. We also construct $\alpha$ and $t$ for which $\Gamma_\alpha \cap (\Gamma_\alpha + t)$ contains only transcendental numbers.

Our approach makes use of digit frequency arguments and a lexicographic characterisation of those $t$ with a unique $\{-1,0,1\}$ $\alpha$-expansion.
\end{abstract}

\maketitle

\section{introduction}\label{sec:1}
 To each $\alpha\in(0,1/2)$ we associate the contracting similarities $f_0(x)=\alpha x$ and $f_{1}(x)=\alpha(x+1).$ The middle $(1-2\alpha)$ Cantor set $\Gamma_\alpha$ is defined to be the unique compact non-empty set satisfying the equation
$$\Gamma_{\alpha} =f_0(\Gamma_\alpha)+f_1(\Gamma_\alpha).$$ It is easy to see that the maps $\{f_0,f_1\}$ satisfy the strong separation condition. Thus $\dim_{H}(\Gamma_{\alpha})=\dim_{B}(\Gamma_{\alpha})=\frac{\log 2}{-\log \alpha},$ where $\dim_{H}$ and $\dim_{B}$ denote the Hausdorff dimension and box dimension respectively.

A natural and well studied question is ``What are the properties of  the intersection $\Gamma_\alpha\cap(\Gamma_\alpha+t)$?" This question has been studied by many authors. We refer the reader to  \cite{Kenyon_Peres_1991, Kraft_1992, Li_Xiao_1998, Kraft_2000, Kong_Li_Dekking_2010} and the references therein for more information. As we now go on to explain, when $\alpha\in(0,1/3]$ the set  $\Gamma_\alpha \cap (\Gamma_\alpha + t)$ is well understood, however when $\alpha\in(1/3,1/2)$ additional difficulties arise.

Note that $\Gamma_\alpha\cap(\Gamma_\alpha+t)\ne\emptyset$ if and only if $t\in\Gamma_\alpha-\Gamma_\alpha$. Thus it is natural to investigate the difference set $\Gamma_{\alpha}-\Gamma_{\alpha},$  which is the self-similar set generated by the iterated function system $\{f_{-1},f_0,f_1\},$ where $f_{-1}(x)=\alpha(x-1)$. Alternatively, one can write
$$ \Gamma_{\alpha}-\Gamma_{\alpha}:=\Big\{\sum_{i=1}^{\infty}\epsilon_{i}\alpha^i: \epsilon_i\in\{-1,0,1\},\,i\geq 1\Big\}.$$
Importantly, for $\alpha\in(0,1/3)$ each $t\in \Gamma_{\alpha}-\Gamma_{\alpha}$ has a unique \emph{$\alpha$-expansion} with \emph{alphabet} $\{-1,0,1\}$, i.e.,  there exists a unique sequence $(t_i)\in\{-1,0,1\}^{\mathbb{N}}$ such that $t=\sum t_i \alpha^{i}.$   When $\alpha=1/3$ there is a countable set of $t$ with precisely two $\alpha$-expansions. These $t$ are well understood and do not pose any real difficulties, thus in what follows we suppress the case where t has two $\alpha$-expansions.

For $\alpha\in(0,1/3]$ let $t\in \Gamma_\alpha - \Gamma_\alpha$ have a unique $\alpha$-expansion $(t_i)$. Then the sequence $(t_i)$ provides a useful description of the set $\Gamma_\alpha \cap (\Gamma_\alpha + t).$ Indeed, we can write (cf.~\cite{Li_Xiao_1998})
\begin{equation}
\label{eq:11}
\G\cap (\G+t)=\Big\{\sum_{i=1}^{\infty}\epsilon_i\alpha^i:\epsilon_i\in \{0,1\}\cap (\{0,1\}+t_i)\Big\}.
\end{equation}
 With this new interpretation many questions regarding the set $\G\cap (\G+t)$ can be reinterpreted and successfully answered using combinatorial properties of the $\alpha$-expansion $(t_i).$

The straightforward description of $\G\cap (\G+t)$ provided by \eqref{eq:11} does not exist for $\alpha\in(1/3,1/2)$ and a generic $t\in \G-\G.$ The set $\G - \G$ is still a self-similar set generated by the transformations $\{f_{-1},f_0.f_1\},$ however this set is now equal to the interval $[\frac{-\alpha}{1-\alpha},\frac{\alpha}{1-\alpha}]$ and the good separation properties that were present in the case where $\alpha\in(0,1/3]$ no longer exist. It is possible that a $t\in \G - \G$ could have many $\alpha$-expansions. In fact it can be shown that Lebesgue almost every $t\in \G - \G$ has a continuum of $\alpha$-expansions (cf.~\cite{Dajani_DeVries_2007, Sidorov_2003, Sidorov_2007}). Thus within the parameter space $(1/3,1/2)$ we are forced to have the following more complicated interpretation of $\G\cap (\G +t)$ (cf.~\cite[Lemma 3.3]{Li_Xiao_1998})
\begin{equation}
\label{eq:12}
\G\cap (\G +t)=\bigcup_{\tilde{t}}\Big\{\sum_{i=1}^{\infty}\epsilon_i\alpha^i:\epsilon_i\in \{0,1\}\cap (\{0,1\}+\tilde{t}_i)\Big\},
\end{equation} where the union is over all $\alpha$-expansions $\tilde{t}=(\tilde{t}_i)$ of $t$.  As stated above, for a generic $t$ this union is uncountable, this makes many questions regarding the set $\G\cap (\G+t)$ intractable. In what follows we focus on the case where $t$ has a unique $\alpha$-expansion. For these $t$ the description of $\G\cap (\G +t)$ given by \eqref{eq:12} simplifies to that given by \eqref{eq:11}.

We now introduce some notation. For $\alpha\in(0,1/2)$ let
 $$\u_\alpha:=\Big\{t\in \G-\G: t \textrm{ has a unique }\alpha\textrm{-expansion w.r.t. the aphabet}~\set{-1,0,1}\Big\}.$$  Within this paper one of our main objects of study is the following set
$$D_{\alpha}:=\Big\{\dim_H(\G\cap(\G+t)): t\in \u_\alpha\Big\}.$$ In particular we will prove the following theorems.

\begin{theorem}
\label{th:11}
There exists a transcendental number $\alpha_{KL}\approx 0.39433\ldots$ such that:
\begin{enumerate}
\item For $\alpha\in (\alpha_{KL},1/2)$ there exists $n^*\in\mathbb{N}$ such that
$$D_{\alpha}=\Big\{0, \frac{\log2}{-\log \alpha}\Big\}\cup\Big\{\frac{\log 2}{ \log\alpha} \sum_{i=1}^{n}\Big(\frac{-1}{2}\Big)^{i}:1\leq n\leq n^*\Big\}.$$
\item
$$D_{\alpha_{KL}}=\Big\{0, \frac{\log2}{-\log \alpha_{KL}},\frac{\log 2}{-3\log \alpha_{KL}}\Big\}\cup\Big\{\frac{\log 2}{ \log\alpha_{KL}} \sum_{i=1}^{n}\Big(\frac{-1}{2}\Big)^{i}:1\leq n< \infty \Big\}.$$
\item $D_{\alpha}$ contains an interval if $\alpha\in(1/3,\alpha_{KL}).$
\end{enumerate}
\end{theorem}

In \cite{Li_Xiao_1998} it was asked ``When $\alpha\in(1/3,1/2)$ what are the possible values of $\dim_H(\G\cap(\G+t))$ for $t\in \G - \G$?"  The following theorem provides a partial solution to this problem.
\begin{theorem}
\label{th:12}
\begin{enumerate}
\item If $\alpha\in(1/3,\frac {3-\sqrt{5}}{2}]$ then $D_{\alpha}=[0,\frac{\log 2}{-\log \alpha}].$
\item If $\alpha\in (\frac{3-\sqrt{5}}{2}, 1/2)$ then $D_{\alpha}$ is a proper subset of $[0,\frac{\log 2}{-\log \alpha}].$
\end{enumerate}

\end{theorem}

Amongst $\G-\G$ a special class of $t$ are those for which $\G\cap (\G+t)$ is a self-similar set. Determining whether  $\G\cap (\G+t)$ is a self-similar set is a difficult problem for a generic $t$ with many $\alpha$-expansions, thus we consider only those $t\in \u_{\alpha}$. Let
$$S_{\alpha}:=\Big\{t\in \u_{\alpha}: \G\cap(\G+t) \textrm{ is a self-similar set} \Big\}.$$

We prove the following result.

\begin{theorem}
\label{th:13}
\begin{enumerate}
\item If $\alpha\in(1/3,\frac{3-\sqrt{5}}{2}]$ then $\{\dim(\G\cap(\G+t) ): t\in S_\alpha\}$ is dense in $[0,\frac{\log 2}{-\log \alpha}].$
\item If $\alpha\in(\frac{3-\sqrt{5}}{2},1/2)$ then $\{\dim(\G\cap(\G+t) ): t\in S_\alpha\}$ is not dense in $[0,\frac{\log 2}{-\log \alpha}].$
\end{enumerate}
\end{theorem}

What remains of this paper is arranged as follows. In Section $2$ we recall the necessary preliminaries from expansions in non-integer bases, and recall an important result of  \cite{Li_Xiao_1998}  that connects the dimension of $\G\cap (\G+t)$ with the frequency of $0$'s in the $\alpha$-expansion $(t_i).$ In Section $3$ we prove Theorem \ref{th:11}, and in Section $4$ we prove Theorem \ref{th:12} and Theorem \ref{th:13}. In Section $5$ we include some examples. We give two examples of an $\alpha\in(1/3,1/2),$ and $t\in \G-\G$ with a continuum of $\alpha$-expansions, for which we can explicitly calculate $\dim_{H}(\G\cap (\G+t))$. The techniques used in our first example can be applied to the more general case where $\alpha$ is the reciprocal of a Pisot number and $t\in \mathbb{Q}(\alpha)$. Our second example demonstrates that it is possible for $t$ to have a continuum of $\alpha$-expansions and for $\G\cap(\G+t)$ to be a self-similar set. Moreover, both of these examples show that it is possible to have $$\dim_{H}(\G+(\G+t))>\sup_{\tilde{t}}\dim_{H}\Big(\Big\{\sum_{i=1}^{\infty}\epsilon_i\alpha^i:\epsilon_i\in \{0,1\}\cap (\{0,1\}+\tilde{t}_i)\Big\}\Big).$$

Our final example demonstrates the existence of $\alpha\in(1/3,1/2)$ and $t\in \G-\G$ for which $\G\cap (\G+t)$ contains only transcendental numbers.

\section{Preliminaries}
Let $M\in \mathbb{N}$ and $\alpha\in [\frac{1}{M+1},1).$ Given $x\in I_{\alpha,M}:=[0,\frac{M\alpha}{1-\alpha}]$ we call a sequence $(\epsilon_{i})\in \{0,\ldots, M\}^{\mathbb{N}}$ an \emph{$\alpha$-expansion} for $x$ with alphabet $\set{0, \cdots,M}$ if $$x=\sum_{i=1}^{\infty}\epsilon_i\alpha^i.$$ This method of representing real numbers was pioneered in the early $1960$'s in the papers of R\'{e}nyi \cite{Renyi_1957} and Parry \cite{Parry_1960}. One aspect of these representations that makes them interesting is that for $\alpha\in(\frac{1}{M+1},1)$ a generic $x\in I_{\alpha,M}$ has many $\alpha$-expansions (cf.~\cite{Dajani_DeVries_2007, Sidorov_2003, Sidorov_2007}).
This naturally leads researchers to study the set of $x\in I_{\alpha,M}$ with a unique $\alpha$-expansion, the so called \emph{univoque set}. We define this set as follows
 $$\u_{\alpha,M}:=\Big\{x\in I_{\alpha,M}: x \textrm{ has a unique }\alpha\textrm{-expansion w.r.t. the aphabet}~\set{0,1,\cdots,M}\Big\}.$$
Accordingly, let $\U_\alpha$ denote the set of corresponding expansions, i.e.,
 $$\U_{\alpha,M}:=\Big\{(\epsilon_i)\in\{0,\ldots,M\}^{\mathbb{N}}: \sum_{i=1}^{\infty}\epsilon_i\alpha^i\in \u_{\alpha,M}\Big\}.$$The sets $\u_{\alpha,M}$ and $\U_{\alpha,M}$ have been studied by many authors. For more information on these sets we refer the reader to \cite{Erdos_Joo_Komornik_1990, Darczy_Katai_1995, Glendinning_Sidorov_2001, DeVries_Komornik_2008, Komornik_2011, Komornik_Kong_Li_2015_1} and the references therein. Before continuing with our discussion of the sets $\u_{\alpha,M}$ and $\U_{\alpha,M}$ we make a brief remark. In the introduction we were concerned with $\alpha$-expansions with digit set $\{-1,0,1\}$, not with a digit set $\{0,\ldots,M\}$. However, all of the result that are stated below for a digit set $\{0,\ldots,M\}$ also hold for any digit set of $M+1$ consecutive integers $\{s,\ldots,s+M\}.$ In particular, statements that are true for the digit set $\{0,1,2\}$ translate to results for the digit set $\{-1,0,1\}$ by performing the substitutions $0\to -1$, $1\to 0$, $2\to 1$.

We now define the lexicographic order   and introduce some notations. Given two finite sequences $\omega=(\omega_1,\ldots,\omega_n),\,\omega'=(\omega_1',\ldots,\omega_n')\in\{0,\ldots,M\}^n,$ we say that $\omega$ is less than $\omega'$ with respect to the lexicographic order, or simply write $\omega\prec \omega'$, if $\omega_1<\omega_1'$ or if there exists $1\leq j<n$ such that $\omega_i=\omega_i'$ for $1\leq i\leq j$ and $\omega_{j+1}<\omega_{j+1}'$. One can also define the relations $\preceq, \succ, \succeq$ in the natural way, and we can extend the lexicographic order   to infinite sequences. We define the \emph{reflection} of a finite/infinite sequence $(\epsilon_i)$ to be $(\overline{\epsilon_i})=(M-\epsilon_i),$ where the underlying $M$ should be obvious from our context. For a finite sequence $\omega=(\omega_1,\ldots,\omega_n)$ we define the finite sequence $\omega^{-}$ to be $(\omega_1,\ldots, \omega_n -1).$ Moreover, we denote the concatenation of $\omega$ with itself $n$ times by $\omega^n$, we also let $\omega^{\infty}$ denote the infinite sequence obtained by indefinitely concatentating $\omega$ with itself.

Given $x\in I_{\alpha,M}$ we define the \emph{greedy} $\alpha$-expansion of $x$ to be the lexicographically largest sequence amongst the $\alpha$-expansions of $x$. We define the \emph{quasi-greedy} $\alpha$-expansion of $x$ to be the lexicographically largest infinite sequence amongst the $\alpha$-expansions of $x$. Here we call a sequence $(\epsilon_i)$ \emph{infinite} if $\epsilon_i\neq 0$ for infinitely many $i$. When studying the sets $\u_{\alpha,M}$ and $\U_{\alpha,M}$ a pivotal role is played by the quasi-greedy $\alpha$-expansion of $1$. In what follows we will denote the quasi-greedy $\alpha$-expansion of $1$ by $(\delta_i(\alpha))$. The importance of the sequence $(\delta_i(\alpha))$ is well demonstrated by the following technical lemma proved in \cite{Parry_1960} (see also, \cite{Erdos_Joo_Komornik_1990, DeVries_Komornik_2008}).

\begin{lemma}
\label{lem:21}
A sequence $(\epsilon_i)$ belongs to $\U_{\alpha,M}$ if and only if the following
two conditions are satisfied:
\begin{align*}
&(\epsilon_{n+i}) \prec (\delta_i(\alpha)) \textrm{ whenever } \epsilon_1\ldots\epsilon_n \neq M^n\\
&(\overline{\epsilon_{n+i}}) \prec (\delta_i(\alpha)) \textrm{ whenever } \epsilon_1\ldots\epsilon_n \neq 0^n
\end{align*}
\end{lemma}
Lemma \ref{lem:21} provides a useful characterisation of the set $\U_{\alpha,M}$ in terms of the sequence $(\delta_i(\alpha))$. The following lemma describes the sequences $(\delta_i(\alpha))$.

\begin{lemma}
\label{lem:22}
Let $M\in\mathbb{N},$ $\alpha\in[\frac{1}{M+1},1)$ and $(\delta_i(\alpha))$ be the quasi-greedy $\alpha$-expansion of $1$. The map $\alpha \to (\delta_i(\alpha))$ is a strictly decreasing bijection from the interval $[\frac{1}{M+1},1)$ onto the set of all infinite sequences
$(\delta_i) \in\{0,\ldots, M\}^{\mathbb{N}}$ satisfying
$$\delta_{k+1}\delta_{k+2} \cdots \preceq\delta_1\delta_2\cdots \textrm{ for all }k \geq 0.$$
\end{lemma}

The following technical result was proved in \cite[Theorem 3.4]{Li_Xiao_1998} for $\alpha\in(0,1/3],$ where importantly every $t$ has a unique $\alpha$-expansion, except for   $\alpha=1/3$ where countably many $t$ have two $\alpha$-expansions. The proof translates over to the more general case where $\alpha\in(1/3,1/2)$ and $t\in \u_{\alpha}.$

\begin{lemma}
\label{lem:23}
Let $\alpha\in(1/3,1/2)$ and $t\in \u_{\alpha},$ then $$\dim_{H}(\G\cap(\G+t))=\frac{\log2}{-\log \alpha}\underline{d}((t_i)),$$
 where $$\underline{d}((t_i)):=\liminf_{n\to \infty} \frac{\#\{1\leq i \leq n: t_i=0\}}{n}.$$
\end{lemma}
 Lemma \ref{lem:23} will be a vital tool in proving Theorems \ref{th:11} and \ref{th:12}. This result allows us to reinterpret Theorems \ref{th:11} and \ref{th:12} in terms of statements regarding the frequency of $0$'s that can occur within an element of $\U_\alpha.$

In what follows, for an infinite sequence $(t_i)\in\{-1,0,1\}^{\mathbb{N}}$ we will use the notation $$ \overline{d}((t_i)):=\limsup_{n\to \infty} \frac{\#\{1\leq i \leq n: t_i=0\}}{n}.$$ When this limit exists, i.e., $\underline{d}((t_i))=\overline{d}((t_i))$, we simply use $d((t_i)).$
For a word $ t_1 \ldots t_n \in\{-1,0,1\}^n$ we will use the notation $$d(t_1\cdots t_n):=\frac{\#\{1\leq i \leq n: t_i=0\}}{n}.$$

\section{Proof of Theorem \ref{th:11}}
In this section we prove Theorem \ref{th:11}. We start by defining the Thue-Morse sequence and its natural generalisation.

Let $(\tau_i)_{i=0}^{\infty}\in\{0,1\}^{\mathbb{N}}$ denote the classical Thue-Morse sequence. This sequence is defined iteratively as follows. Let $\tau_0=0$ and if $\tau_i$ is defined for some $i\geq 0$, set $\tau_{2i}=\tau_i$ and $\tau_{2i+1}=1-\tau_i.$ Then the sequence $(\tau_i)_{i=0}^{\infty}$ begins with
 $$0110\, 1001\, 1001\, 0110\, 1001\, 0110 0110\ldots$$ For more on this sequence we refer the reader to \cite{Allouche_Shallit_1999}. Within expansions in non-integer bases the sequence $(\tau_i)_{i=0}^{\infty}$ is important for many reasons. In \cite{Komornik_Loreti_1998} Komornik and Loreti proved that the unique $\alpha$ for which $(\delta_{i}(\alpha))=(\tau_i)_{i=1}^{\infty}$ is the largest $\alpha\in(1/2,1)$ for which $1$ has a unique $\alpha$-expansion. This $\alpha$ has since become known as the \emph{Komornik-Loreti constant}. Interesting connections between the size of $\u_{\alpha}$ and the Komornik Loreti constant were made in \cite{Glendinning_Sidorov_2001}. Using the Thue-Morse sequence we define a new sequence $(\lambda_i)\in \{-1,0,1\}^{\mathbb{N}}$ as follows $$(\lambda_i)_{i=1}^{\infty}=(\tau_i-\tau_{i-1})_{i=1}^{\infty}.$$ We denote the unique $\alpha\in(1/2, 1)$ for which $\sum_{i=1}^{\infty}(1+\lambda_i)\alpha^i=1$ by $\alpha_{KL}$.
Our choice of subscript is because the constant $\alpha_{KL}$ is a type of generalised Komornik-Loreti constant. This number is transcendental (cf.~\cite{Komornik_Loreti_2002}) and is approximately $0.39433.$
This sequence satisfies the property
\begin{equation}
\label{eq:31}
\begin{split}
&\lambda_1=1,\quad\quad\lambda_{2^{n+1}}=1-\lambda_{2^n};\\
&\lambda_{2^n+i}=- \lambda_i \quad \textrm{ for any}~~ 1\leq i<2^n.
\end{split}
\end{equation}
This property can be deduced directly from \cite[Lemma 5.2]{Komornik_Loreti_2002}.
So, the sequence $(\lambda_i)_{i=1}^\infty$ starts at
\[
10\,(-1)1\,(-1)010\;(-1)01(-1)\,10(-1)1\cdots.
\] It will be useful when it comes to determining the frequency of zeros within certain sequences.

To each $n\in\mathbb{N}$ we associate the finite sequence $w_n=\lambda_1\cdots \lambda_{2^{n}}.$  By (\ref{eq:31}) the following property of $\omega_n$ can be verified.
\begin{equation}
\label{eq:32}
 w_{n+1}^{-}=w_n\overline{w_n}.
\end{equation}
 Here the reflection of   $w_n$ w.r.t. the digit set $\{-1,0,1\}$ is defined by $\overline{w_n}:=(-\lambda_1)(-\lambda_2)\cdots(-\lambda_{2^n}).$

We now prove two lemmas that allow us to prove statements $(1)$ and $(2)$ from Theorem \ref{th:11}.
\begin{lemma}
\label{lem:31}
For $n\geq 2$ the following inequalities hold:
\begin{equation}
\label{eq:33}
\#\{1\leq i\leq 2^n: \lambda_i=0\}=2\#\{1\leq i\leq 2^{n-1}: \lambda_i=0\} - 1 \textrm{ if }n\textrm{ is even;}
\end{equation}
\begin{equation}
\label{eq:34}
\#\{1\leq i\leq 2^n: \lambda_i=0\}=2\#\{1\leq i\leq 2^{n-1}: \lambda_i=0\} + 1 \textrm{ if }n\textrm{ is odd}.
\end{equation}
Moreover
\begin{equation}
\label{eq:35}
d(w_n)=-\sum_{i=1}^{n}\Big(\frac{-1}{2}\Big)^{i}
\end{equation} for all $n\in \mathbb{N}.$
\end{lemma}

\begin{proof}
We begin by observing that $w_1=10,$ so $d(w_{1})=1/2$ and (\ref{eq:35}) holds for $n=1$. We now show that \eqref{eq:33} and \eqref{eq:34} imply \eqref{eq:35} via an inductive argument. Let us assume \eqref{eq:35} is true for odd $N\in\mathbb{N}$.  Then
\begin{align*}
d(w_{N+1})&= \frac{\#\{1\leq i\leq 2^{N+1}: \lambda_{i}=0\}}{2^{N+1}}\\
&= \frac{2\#\{1\leq i\leq 2^{N}: \lambda_i=0\} - 1}{2^{N+1}}\\
&=d(w_N)-\frac{1}{2^{N+1}}\\
&=-\sum_{i=1}^{N+1}\Big(\frac{-1}{2}\Big)^i
\end{align*}In our second equality we used \eqref{eq:33}. The case where $N$ is even is done similarly. Proceeding inductively we may conclude that \eqref{eq:35} holds assuming \eqref{eq:33} and \eqref{eq:34}.

It remains to show \eqref{eq:33} and \eqref{eq:34} hold. For $n=1$ we know that $w_1=10,$ \eqref{eq:32} therefore implies that the last digit of $w_2$ equals $1.$ What is more, repeatedly applying \eqref{eq:32} we see that the last digit of $w_n$ equals $0$ if $n$ is odd, and equals $1$ if $n$ is even. Property \eqref{eq:31} implies that   $\lambda_{2^n+i}=0$ if $\lambda_{i}=0$ for any $1\le i<2^n$. Therefore, when $n$ is even we see that
\begin{align*}
\#\{1\leq i\leq 2^n: \lambda_i=0\}&=\#\{1\leq i\leq 2^{n-1}: \lambda_i=0\}+\#\{2^{n-1}+1\leq i\leq 2^n: \lambda_i=0\}\\
& = \#\{1\leq i\leq 2^{n-1}: \lambda_i=0\} +  \#\{1\leq i\leq 2^{n-1}: \lambda_i=0\} -1 \\
&= 2\#\{1\leq i\leq 2^{n-1}: \lambda_i=0\} -1.
\end{align*} Thus \eqref{eq:33} is proved. Equation \eqref{eq:34} is proved similarly.
\end{proof}

Lemma \ref{lem:31} determines the frequency of $0$'s within the finite sequences $w_n$. For our proof of Theorem \ref{th:11} we also need to know the frequency of $0$'s within the sequence $(\lambda_i)_{i=1}^{\infty}.$

\begin{lemma}
\label{lem:32}
$$d((\lambda_i))=-\sum_{i=1}^{\infty}\Big(\frac{-1}{2}\Big)^{i}=\frac{1}{3}.$$
\end{lemma}
\begin{proof}
Let us begin by fixing $\varepsilon>0.$ Let $N\in\mathbb{N}$ be sufficiently large such that
\begin{equation}
\label{eq:36}
\Big|\frac{-\sum_{i=1}^{n}(-1/2)^i}{1/3}-1\Big|<\varepsilon
\end{equation} for all $n\geq N$. Now let us pick $N'\in\mathbb{N}$ large enough  such that
\begin{equation}
\label{eq:37}
\frac{\sum_{j=0}^{N-1}  2^j}{N'}<\varepsilon
\end{equation}

Let $n\geq N'$ be arbitrary and write $n=\sum_{j=0}^{k}\epsilon_{j}2^{j},$ where we assume $\epsilon_k=1$. By splitting $(\lambda_i)_{i=1}^n$ into its first $2^{k}$ digits, then the next $2^{k-1}$ digits, then the next $2^{k-2}$ digits, etc, we obtain:
\begin{align}
\label{eq:38}
\frac{\#\{1\leq i\leq n:\lambda_i=0\}}{n}&=\frac{\#\{1\leq i\leq 2^{k}:\lambda_i=0\}}{n}\\
&+\sum_{l=0}^{k-1}\frac{\#\{\sum_{j=k-l}^{k}\epsilon_j 2^j+1 \leq i \leq \sum_{j=k-l-1}^{k}\epsilon_j 2^j:\lambda_i=0\}}{n}.\nonumber
\end{align}
By repeatedly applying \eqref{eq:31} we see
\begin{equation}\label{eq:39}
\begin{split}
&\quad~\#\{1 \leq i \leq \epsilon_{k-l-1} 2^{k-l-1}:\lambda_i=0\}\\
&= \#\{\epsilon_{k-l}2^{k-l}+1 \leq i \leq \epsilon_{k-l}2^{k-l}+ \epsilon_{k-l-1} 2^{k-l-1}:\lambda_i=0\} \\
& =\cdots \\
& = \#\Big\{\sum_{j=k-l}^{k}\epsilon_j 2^j+1 \leq i \leq \sum_{j=k-l-1}^{k}\epsilon_j 2^j:\lambda_i=0\Big\}
\end{split}
\end{equation}
Substituing \eqref{eq:39} into \eqref{eq:38} we obtain
\begin{align*}
 \frac{\#\{1\leq i\leq n:\lambda_i=0\}}{n}&=\frac{\#\{1\leq i\leq 2^{k}:\lambda_i=0\}}{n}\\
&+\sum_{l=0}^{k-1}\frac{\#\{1 \leq i \leq \epsilon_{k-l-1} 2^{k-l-1}:\lambda_i=0\}}{n}.\nonumber
\end{align*}
By ignoring lower order terms and applying Lemma \ref{lem:31}, \eqref{eq:36}  and \eqref{eq:37} we obtain the lower bound
\begin{align*}
\label{lowerbound}
\frac{\#\{1\leq i\leq n:\lambda_i=0\}}{n}&\geq \frac{\#\{1\leq i\leq 2^{k}:\lambda_i=0\}}{n}+\sum_{l=0}^{k-N-1}\frac{\#\{1 \leq i \leq \epsilon_{k-l-1} 2^{k-l-1}:\lambda_i=0\}}{n}\\
& \geq \frac{(1-\varepsilon)}{3}\Big(\frac{ 2^{k}}{n}+\sum_{l=0}^{k-N-1}\frac{\epsilon_{k-l-1} 2^{k-l-1}}{n}\Big)\\
 &= \frac{(1-\varepsilon)}{3}\Big(\frac{\sum_{j=0}^{k}\epsilon_j2^j-\sum_{j=0}^{N-1}\epsilon_j2^j}{n}\Big)\\
&\ge \frac{(1-\varepsilon)}{3}\Big(1-\frac{\sum_{j=0}^{N-1}2^j}{n}\Big)\\
& \geq \frac{(1-\varepsilon)^2}{3}.
\end{align*}
As $\varepsilon>0$ was arbitrary this implies $\underline{d}((\lambda_i))\ge 1/3.$ By a similar argument it can be shown that $\overline{d}((\lambda_i))\le 1/3.$ Thus $d((\lambda_i))=1/3.$
\end{proof}
Statements $(1)$ and $(2)$ from Theorem \ref{th:11} follow from Lemma \ref{lem:23},  Lemma \ref{lem:31}, and Lemma \ref{lem:32}, when combined with the following results from \cite[Lemma 4.12]{Kong_Li_Dekking_2010}.

\begin{lemma}
\label{lem:33}
Let $\alpha\in(\alpha_{KL},1/2),$ then there exists $n^{*}\in\mathbb{N}$ such that every element of $\U_{\alpha}\setminus\set{(-1)^\f, 1^\f}$ ends with one of
$$(0)^{\infty}, (w_1\overline{w_1})^{\infty}, \ldots, (w_{n^{*}}\overline{w_{n^{*}}})^{\infty}.$$
\end{lemma}

\begin{lemma}
\label{lem:34}
Each element of $\U_{\alpha_{KL}}\setminus\set{(-1)^\infty, 1^\infty}$ is either eventually periodic with period contained in $$(0)^{\infty}, (w_1\overline{w_1})^{\infty}, (w_2\overline{w_2})^\infty, \ldots, $$ or ends with a sequence of the form $$(w_0\overline{w_{0}})^{k_{0}}(w_0\overline{w_{i_1'}})^{k_{0}'}(w_{i_1}\overline{w_{i_1}})^{k_{1}}(w_{i_{1}}\overline{w_{i_{2}'}})^{k_{1}'}\cdots(w_{i_n}\overline{w_{i_{n}}})^{k_{n}}(w_{i_n}\overline{w_{i_{n+1}'}})^{k_{n}'}\cdots,$$ and its reflection, where $k_n\geq 0,$ $k_n'\in\{0,1\}$ and
\[0<i_1'\leq  i_1<i_2'\leq i_2<\cdots \leq i_n< i_{n+1}' \leq i_{n+1}<\cdots.\]
\end{lemma}
By Lemmas \ref{lem:23},  \ref{lem:31} and \ref{lem:33} we may conclude
 $$D(\alpha)=\Big\{0, \frac{\log2}{-\log \alpha}\Big\}\cup\Big\{\frac{\log 2}{  \log \alpha} \sum_{i=1}^{n}\Big(\frac{-1}{2}\Big)^i:1 \leq n \leq n^*\Big\}$$ for some $n^*\in\mathbb{N}$ for $\alpha\in(\alpha_{KL},1/2)$. Whilst at the constant $\alpha_{KL}$ by  Lemmas \ref{lem:23},   \ref{lem:31},   \ref{lem:32} and   \ref{lem:34} we have
$$D(\alpha_{KL})=\Big\{0, \frac{\log2}{-\log \alpha_{KL}},\frac{\log 2}{-3\log\alpha_{KL}}\Big\}\cup\Big\{\frac{\log 2}{  \log \alpha_{KL}} \sum_{i=1}^{n}\Big(\frac{-1}{2}\Big)^i:1 \leq n < \infty\Big\}.$$
Thus statements $(1)$ and $(2)$ from Theorem \ref{th:11} hold. It remains to prove statement $(3).$

We start by introducing the following finite sequences. Let
\begin{equation}
\label{eq:310}
\zeta_{n}=0\lambda_1\cdots \lambda_{2^{n}-1}\textrm{ and }\eta_{n}=(-1)\lambda_{1}\ldots\lambda_{2^{n}-1}.
\end{equation} The following result was proved in \cite{Kong_Li_Dekking_2010}.

\begin{lemma}
\label{lem:35}
Let $\alpha\in(1/3,\alpha_{KL}),$ then there exists $n\in\mathbb{N}$ such that $\U_{\alpha}$ contains the subshift of finite type over the alphabet $\mathcal{A}=\{\zeta_n,\eta_n,\overline{\zeta_n},\overline{\eta_n\}}$ with transition matrix
\[ A=\left( \begin{array}{cccc}
0 & 1 & 1 & 0 \\
0 & 0 & 1 & 0 \\
1 & 0 & 0 & 1 \\
1 & 0 & 0 & 0 \end{array} \right).\]
\end{lemma}

\begin{proof}[Proof of   Theorem \ref{th:11} (3)]
Let $\alpha\in (1/3,\alpha_{KL})$ and let $n$ be as in Lemma \ref{lem:35}. So $\U_{\alpha}$ contains the subshift of finite type determined by the alphabet $\mathcal{A}$ and the transition matrix $A$. On closer examination we see that this subshift of finite type allows the free concatentation of the words $\omega_1=\zeta_n \overline{\zeta_n}$ and  $\omega_2=\zeta_n \eta_n\overline{\zeta_n}.$ Importantly $d(\omega_1)< d(\omega_2)$ by \eqref{eq:310}.  For any $c\in[ d(\omega_1),d(\omega_2)]$ we can pick a sequence of integers $k_{1}, k_{2},\ldots$ such that the sequence $(\epsilon_i)=\omega_{1}^{k_1}\omega_{2}^{k_{2}}\omega_{1}^{k_{3}}\omega_2^{k_{4}}\ldots$ satisfies $d((\epsilon_i))=c$. Thus by Lemma \ref{lem:23} the set $D(\alpha)$ contains the interval $[\frac{\log 2}{-\log \alpha}d(\omega_1),\frac{\log 2}{-\log \alpha}d(\omega_2)]$ and our proof is complete.
\end{proof} Appealing to standard arguments from multifractal analysis we could in fact show that for any $c\in(\frac{\log 2}{-\log \alpha}d(\omega_1),\frac{\log 2}{-\log \alpha}d(\omega_2))$ there exists a set of positive Hausdorff dimension within $\u_{\alpha}$ with frequency $c$.

\section{Proof of Theorem \ref{th:12} and Theorem \ref{th:13}}
We start this section by proving Theorem \ref{th:12}. Theorem \ref{th:13} will follow almost immediately as a consequence of the arguments used in the proof of Theorem \ref{th:12}. To prove Theorem \ref{th:12} we rely on the lexicographic description of $\U_{\alpha}$ and $(\delta_i(\alpha))$ given in Section $2$. We take this opportunity to again emphasise that the preliminary results that hold in Section $2$ for the alphabet $\{0,1,2\}$ have an obvious analogue that holds for the digit set $\{-1,0,1\}$.

It is instructive here to state our analogue of the quasi greedy $\alpha$-expansion   of $1$ when $\alpha=\frac{3-\sqrt{5}}{2}$ and our digit set is $\{-1,0,1\}$. A straightforward calculation proves that this analogue satisfies
\begin{equation}\label{eq:41}
\Big(\delta_i\Big(\frac{3-\sqrt{5}}{2}\Big)\Big)= 1(0)^{\infty}.
\end{equation}

We split our proof of Theorem \ref{th:12} into two lemmas.

\begin{lemma}
\label{lem:41}
Let $\alpha\in(\frac{3-\sqrt{5}}{2},1/2),$ then there exists $n\in \mathbb{N}$ such that any element of $\U_\alpha$ cannot contain the sequence $1(0)^{n}$ or $(-1)(0)^{n}$ infinitely often.
\end{lemma}
\begin{proof}
Suppose $\alpha\in(\frac{3-\sqrt{5}}{2},1/2)$. Then by Lemma \ref{lem:22} and (\ref{eq:41}) we have
\begin{equation}
\label{eq:42}
(\delta_i(\alpha))\prec (1(0)^{\infty}).
\end{equation} For any $\alpha\in(1/3,1/2)$ we have $\delta_1(\alpha)=1$. Therefore by \eqref{eq:42} there exists $k\geq 0$ such that $(\delta_i(\alpha))$ begins with the word $1(0)^{k}(-1)$. If a sequence $(\epsilon_i)\in \U_\alpha$ contained the sequence $1(0)^{k+1}$ infinitely often, then it is a consequence of Lemma \ref{lem:21} for the digit set $\{-1,0,1\}$ that the following lexicographic inequalities would have to hold
\begin{equation}
\label{eq:43}
(-1)(0)^k1\preceq 1(0)^{k+1} \preceq 1(0)^{k}(-1).
\end{equation}
 Clearly the right hand side of \eqref{eq:43} does not hold, therefore $1(0)^{k+1}$ cannot occur infinitely often. Similarly, one can show that $(-1)(0)^{k+1}$ cannot occur infinitely often by considering the left hand side of \eqref{eq:43}.
\end{proof}

\begin{lemma}
\label{lem:42}
If $\alpha\in(1/3,\frac{3-\sqrt{5}}{2}]$ then for any sequence of natural numbers $(n_i)$ the sequence $$(1(-1))^{n_{1}} \,0 ^{n_{2}}\,(1(-1))^{n_{3}}\,0^{n_4}\cdots$$ is contained in $\U_{\alpha}$.
\end{lemma}

\begin{proof}
Fix a sequence of natural numbers $(n_i)$. It is a consequence of Lemma \ref{lem:21} and Lemma \ref{lem:22} that $\U_{\frac{3-\sqrt{5}}{2}}\subset \U_{\alpha}$ for all $\alpha\in(1/3,\frac{3-\sqrt{5}}{2}).$ Therefore it suffices to show that the sequence $$(\epsilon_i)_{i=1}^{\infty}:=(1(-1))^{n_{1}}\,0^{n_{2}}\,(1(-1))^{n_{3}}\,0^{n_{4}}\cdots$$ is contained in $\U_{\frac{3-\sqrt{5}}{2}}.$  For all $n\geq 0$ the following lexicographic inequalities hold
$$(-1)(0)^{\infty}\prec (\epsilon_i)_{i=n+1}^{\infty}\prec 1(0)^{\infty}.$$ Applying Lemma \ref{lem:21} we see that $(\epsilon_i)\in {\U}_{\frac{3-\sqrt{5}}{2}}$ and our proof is complete.
\end{proof}

\begin{proof}[Proof of Theorem \ref{th:12}]
Let $\alpha\in(\frac{3-\sqrt{5}}{2},1/2)$ and let $N\in \mathbb{N}$ be as in Lemma \ref{lem:41}.  Now let us pick $a\in (\frac{N}{N+1},1).$ Any $(t_i)\in \U_\alpha$ with $\underline{d}((t_i))=a$ must contain either the sequence $1(0)^N$ infinitely often or $(-1)(0)^{N}$ infinitely often. By Lemma \ref{lem:41} this is not possible. Thus by Lemma \ref{lem:23} the set $D_\alpha$ is a proper subset of $[0,\frac{\log2}{-\log \alpha}]$ and statement $(2)$ of Theorem \ref{th:12} holds.

By Lemma \ref{lem:23} it remains to show that for any $\alpha\in(1/3,\frac{3-\sqrt{5}}{2}]$ and $a\in [0,1]$ there exists $(t_i)\in \U_{\alpha}$ such that $\underline{d}((t_i))=a.$ The existence of such a $(t_i)$ now follows from Lemma \ref{lem:42} by making an appropriate choice of $(n_i).$
\end{proof}

We now prove Theorem \ref{th:13}. To prove this theorem we require the following technical characterisation of $S_{\alpha}$ from \cite[Theorem 3.2]{Kong_Li_Dekking_2010}. We recall that an infinite sequence $(\omega_i)\in\{0,1\}^{\mathbb{N}}$ is called \emph{strongly eventually periodic} if $(\omega_i)=IJ^\infty$, where $I,J$ are two finite words of the same length and $I\preceq J.$ Clearly, a periodic sequence is strongly eventually periodic.

\begin{proposition}
\label{prop:43}
$t\in S_{\alpha}$ if and only if $(1-|t_i|)_{i=1}^{\infty}$ is strongly eventually periodic.
\end{proposition}

\begin{proof}[Proof of Theorem \ref{th:13}]
Statement $(2)$ of Theorem \ref{th:13} follows from the proof of Theorem \ref{th:12}. It is a consequence of our proof that for $\alpha\in(\frac{3-\sqrt{5}}{2},1/2)$ there exists $\epsilon>0$ such that $\underline{d}((t_i))\notin(1-\epsilon,1)$ for all $(t_i)\in \U_{\alpha}$. This statement when combined with Lemma \ref{lem:23} implies statement $(2)$ of Theorem \ref{th:13}.

To prove statement $(1)$ we remark that for any $\alpha\in(1/3,\frac{3-\sqrt{5}}{2}]$ and $n_1,\ldots, n_j\in\mathbb{N},$ the sequence
$$(t_i)=((1(-1))^{n_{1}}\,0^{n_{2}}\,(1(-1))^{n_3}\cdots (1(-1))^{n_{j-1}}\,0^{n_{j}})^{\infty}$$ is contained in $\U_\alpha.$ The sequence $(1-|t_i|)$ is strongly eventually periodic, therefore by Proposition \ref{prop:43} the corresponding $t$ is contained in $S_{\alpha}.$ For any $a\in [0,1]$ and $\epsilon >0,$ we can pick $n_{1},\ldots,n_j\in\mathbb{N}$ such that $|d((t_i))-a|<\epsilon.$ Applying Lemma \ref{lem:23} we may conclude that statement (1) of Theorem \ref{th:13} holds.
\end{proof}

\section{Examples}
We end our paper with some examples. We start with two examples of an $\alpha\in(1/3,1/2),$ and a $t\in \G-\G$ with a continuum of $\alpha$-expansions for which the Hausdorff dimension of $\G\cap (\G+t)$ is explicitly calculable. The approach given in the first example applies more generally to $\alpha$ the reciprocal of a Pisot number and $t\in \mathbb{Q}(\alpha).$ Our second example demonstrates that it is possible for $t$ to have a continuuum of $\alpha$-expansions and for $\G\cap (\G+t)$ to be a self-similar set.

\begin{example}
Let $\alpha=0.449\ldots$ be the unique real root of $2x^3+2x^2+x-1=0.$ Consider $t=\sum_{i=1}^{\infty}(-\alpha)^i.$ For this choice of $\alpha$ the set of $\alpha$-expansions of $t$ is equal to the allowable sequences of edges in Figure \ref{fig1} that start at the point $((-1)1)^{\infty}$.

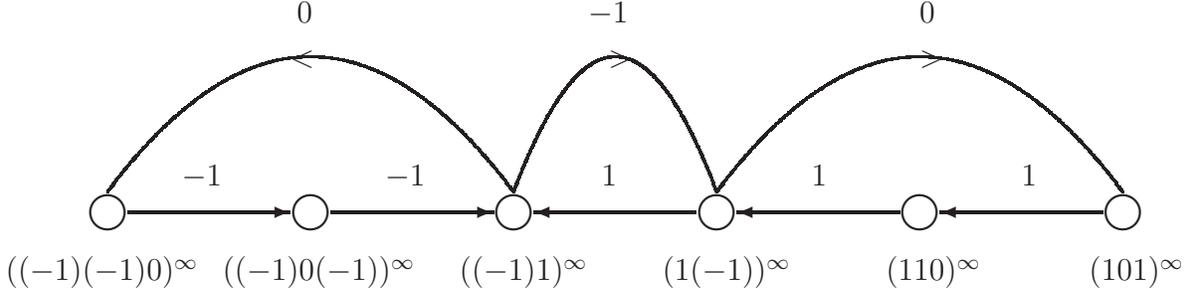
\begin{figure}
\setlength{\unitlength}{0.09cm}
\thicklines
\begin{picture}(190,45)(0,-10)
\put(20,2){\circle{5}}
\put(50,2){\circle{5}}
\put(80,2){\circle{5}}
\put(110,2){\circle{5}}
\put(140,2){\circle{5}}
\put(170,2){\circle{5}}

\put(102,-8){$(1(-1))^{\infty}$}
\put(72,-8){$((-1)1)^{\infty}$}
\put(135,-8){$(110)^{\infty}$}
\put(165,-8){$(101)^{\infty}$}
\put(5,-8){$((-1)(-1)0)^{\infty}$}
\put(37,-8){$((-1)0(-1))^{\infty}$}

\qbezier(80,5)(95,45)(110,5)
\put(94,23.5){$>$}
\qbezier(110,5)(140,45)(170,5)
\put(140,23.5){$>$}
\qbezier(20,5)(50,45)(80,5)
\put(47,23.5){$<$}

\put(23,2){\vector(1,0){24}}

\put(53,2){\vector(1,0){24}}

\put(107,2){\vector(-1,0){24}}

\put(137,2){\vector(-1,0){24}}

\put(167,2){\vector(-1,0){24}}

\put(31,6){$-1$}
\put(61,6){$-1$}
\put(93,6){$1$}
\put(124,6){$1$}
\put(155,6){$1$}
\put(91,30){$-1$}
\put(140,30){$0$}
\put(48,30){$0$}

\end{picture}
\caption{A graph generating all $\alpha$-expansions of $t=\sum_{i=1}^{\infty}(-\alpha)^i.$}
    \label{fig1}

\end{figure}

Using \eqref{eq:12} we see that $\G\cap (\G+t)$ coincides with those numbers $\sum_{i=1}^{\infty}\epsilon_i\alpha^i$ where $(\epsilon_i)$ is a sequence of allowable edges in Figure \ref{fig2} that start at $((-1)1)^{\infty}$.

\begin{figure}
\setlength{\unitlength}{0.09cm}
\thicklines
\begin{picture}(190,45)(0,-10)
\put(20,2){\circle{5}}
\put(50,2){\circle{5}}
\put(80,2){\circle{5}}
\put(110,2){\circle{5}}
\put(140,2){\circle{5}}
\put(170,2){\circle{5}}

\put(102,-8){$(1(-1))^{\infty}$}
\put(72,-8){$((-1)1)^{\infty}$}
\put(134,-8){$(110)^{\infty}$}
\put(163,-8){$(101)^{\infty}$}
\put(7,-8){$((-1)(-1)0)^{\infty}$}
\put(37,-8){$((-1)0(-1))^{\infty}$}

\qbezier(80,5)(95,45)(110,5)
\put(94,23.5){$>$}
\qbezier(110,5)(140,45)(170,5)
\put(140,23.5){$>$}
\qbezier(20,5)(50,45)(80,5)
\put(47,23.5){$<$}

\put(23,2){\vector(1,0){24}}

\put(53,2){\vector(1,0){24}}

\put(107,2){\vector(-1,0){24}}

\put(137,2){\vector(-1,0){24}}

\put(167,2){\vector(-1,0){24}}

\put(33,6){$0$}
\put(63,6){$0$}
\put(93,6){$1$}
\put(125,6){$1$}
\put(155,6){$1$}
\put(93,30){$0$}
\put(138,30){$0/1$}
\put(46,30){$0/1$}

\end{picture}
\caption{A graph generating $\G\cap (\G+t)$}
    \label{fig2}

\end{figure}
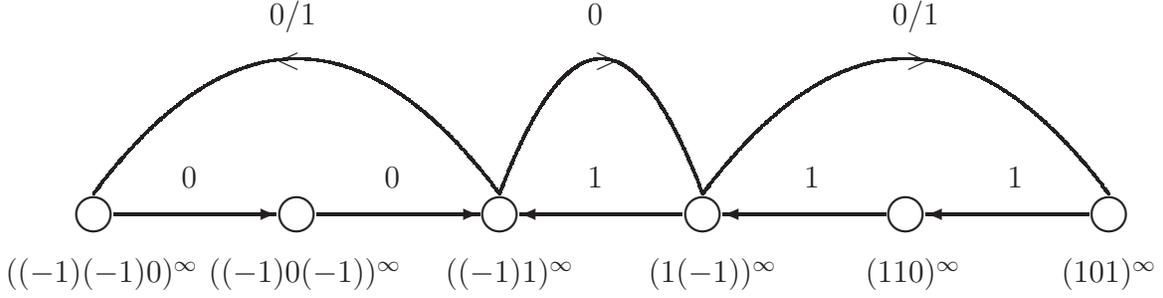

We let
 $$C_n:=\Big\{(\epsilon_i)_{i=1}^n\in\{0,1\}^n: \Big[\sum_{i=1}^{n}\epsilon_i\alpha^i,\sum_{i=1}^{n}\epsilon_i\alpha^i+ \frac{\alpha^{n+1}}{1-\alpha}\Big]\bigcap (\G\cap (\G+t))\neq\emptyset\Big\}.$$ Given $\delta_1\cdots\delta_m\in C_m$ we let
 $$C_{n}(\delta_1\cdots\delta_m):=\Big\{(\epsilon_i)_{i=1}^{m+n}\in C_{m+n}: (\epsilon_1,\ldots,\epsilon_m)=(\delta_1,\ldots, \delta_m)\Big\}.$$ Making use of standard arguments for transition matrices it can be shown that there exists $c>0$ such that
\begin{equation}\label{eq:51}
\frac{\lambda^{n}}{c}\leq \#C_{n}\leq c \lambda^n \textrm{ and }\frac{\lambda^{n}}{c}\leq \#C_{n}(\delta_1\cdots\delta_m)\leq c \lambda^n,
\end{equation} for any $\delta_1\cdots\delta_m\in C_m.$ Here $\lambda\approx 1.69562\ldots$ is the unique maximal eigenvalue of the matrix

\[ A=\left( \begin{array}{cccccc}
0 & 1 & 0 & 0 & 0 & 0\\
0 & 0 & 1 & 0 & 0 & 0\\
2 & 0 & 0 & 1 & 0 & 0\\
0 & 0 & 1 & 0 & 0 & 2\\
0 & 0 & 0 & 1 & 0 & 0\\
0 & 0 & 0 & 0 & 1 & 0\\ \end{array} \right).\]
In the following we will show that
\begin{equation}
  \label{eq:52}
  \dim_H(\Gamma_\alpha\cap(\Gamma_\alpha+t))=\frac{\log\lambda}{-\log\alpha}\approx 0.644297.
\end{equation}
In fact we   show that  $0<\mathcal{H}^{\frac{\log \lambda}{-\log \alpha}}(\G\cap (\G+t))<\infty.$
By (\ref{eq:51}) the upper bound follows from the following straightforward argument:
\begin{align*}\mathcal{H}^{\frac{\log \lambda}{-\log \alpha}}(\G\cap (\G+t))&\leq \liminf_{n\to \infty}\sum_{(\epsilon_i)\in C_n} Diam\Big(\big[\sum_{i=1}^{n}\epsilon_i\alpha^i,\sum_{i=1}^{n}\epsilon_i\alpha^i+ \frac{\alpha^{n+1}}{1-\alpha}\big]\Big)^{\frac{\log \lambda}{-\log \alpha}}\\
&\leq c\lambda^n \Big(\frac{\alpha^{n+1}}{1-\alpha}\Big)^{\frac{\log \lambda}{-\log \alpha}}\\
& <\infty
\end{align*}In what follows we use the notation $\mathcal{I}_n$ to denote the basic intervals corresponding to the elements of $C_n,$ and  $\mathcal{I}_{n}(\delta_1\cdots\delta_m)$ to denote the basic intervals corresponding to elements of $C_{n}(\delta_1\cdots\delta_m).$

The proof that $\mathcal{H}^{\frac{\log \lambda}{-\log \alpha}}(\G\cap (\G+t))>0$ is based upon arguments given in \cite{Baker_2014_1} and Example $2.7$ from \cite{Falconer_1990}. Let $\{U_j\}_{j=1}^{\infty}$ be an arbitrary cover of $\G\cap (\G+t).$ Since $\G\cap (\G+t)$ is compact we can assume that $\{U_j\}_{j=1}^{p}$ is a finite cover.  For each $U_j$ there exists $l(j)\in\mathbb{N}$ such that $\frac{\alpha^{l(j)+1}}{1-\alpha}<Diam(U_j)\leq \frac{\alpha^{l(j)}}{1-\alpha}.$ This implies that $U_j$ intersects at most two elements of $ \mathcal{I}_{l(j)} $.
This means that for each $j$ there exists at most two codes $(\epsilon_{1},\ldots,\epsilon_{l(j)}),(\epsilon_{1}',\ldots,\epsilon_{l(j)}')\in C_{l(j)}$ such that
\[U_j\cap \Big[\sum_{i=1}^{l(j)}\epsilon_i \alpha^i, \sum_{i=1}^{l(j)}\epsilon_i \alpha^i +\frac{\alpha^{l(j)}}{1-\alpha}\Big]\neq \emptyset\quad\textrm{ and}\quad U_j\cap \Big[\sum_{i=1}^{l(j)}\epsilon_i' \alpha^i, \sum_{i=1}^{l(j)}\epsilon_i' \alpha^i +\frac{\alpha^{l(j)}}{1-\alpha}\Big]\neq \emptyset.\]
 Without loss of generality we may assume that $U_j$ always intersects at least one element of $\mathcal{I}_{l(j)}$. Since $\{U_i\}_{i=1}^p$ is a finite cover there exists $J\in\mathbb{N}$ such that $\alpha^{J}<Diam(U_i)$ for all $i$.  By (\ref{eq:51}) the following inequalities hold by counting arguments:

\begin{align*}
\frac{\lambda^J}{c}\leq \#C_J & \leq \sum_{j=1}^p \#\Big\{(\epsilon_i)\in C_J: \Big[\sum_{i=1}^{J}\epsilon_i\alpha^i,\sum_{i=1}^{J}\epsilon_i\alpha^i+ \frac{\alpha^{J+1}}{1-\alpha}\Big]\cap U_j \neq \emptyset\Big\}\\
&\leq \sum_{j=1}^p \#C_{J-l(j)}(\epsilon_{1}\cdots\epsilon_{l(j)}) +\sum_{j=1}^m \#C_{J-l(j)}(\epsilon_{1}'\cdots\epsilon'_{l(j)}) \\
&\leq 2c\sum_{j=1}^{p}\lambda^{J-l(j)} \\
&\leq 2c\sum_{j=1}^{p}\lambda^J\cdot \alpha^{-l(j)\frac{\log \lambda}{-\log \alpha}}.
\end{align*}
Cancelling through by $\lambda^J$ we obtain $(2c^2)^{-1} \leq \sum_{j=1}^{p} \alpha^{-l(j)\frac{\log \lambda}{-\log \alpha}}.$ Since $Diam(U_j)$ is $\alpha^{l(j)}$ up to a constant term we may deduce that $\sum_{j=1}^{p}Diam(U_j)^{\frac{\log \lambda}{-\log \alpha}}$ can be bounded below by a strictly positive constant that does not depend on our choice of cover. This in turn implies $\mathcal{H}^{\frac{\log \lambda}{-\log \alpha}}(\G\cap (\G+t))>0.$

By \eqref{eq:12} we know that
\begin{equation}
\label{eq:53}
\dim_{H}(\G\cap (\G+t)) \geq \sup_{\tilde{t}} \dim_{H}\Big\{ \sum_{i=1}^{\infty}\epsilon_i\alpha^i:\epsilon_i\in \{0,1\}\cap (\{0,1\}+\tilde{t}_i)\Big\},
\end{equation} where the supremum is over all $\alpha$-expansions of $t$. If $t$ has countably many $\alpha$-expansions, then by the countable stability of the Hausdorff dimension we would have equality in \eqref{eq:53}. In the case where $t$ has a continuum of $\alpha$-expansions it is natural to ask whether equality persists. This example shows that this is not the case. Upon examination of Figure \ref{fig1} we see that any $\alpha$-expansion  of $((-1)1)^{\infty}$ satisfies $\underline{d}((t_i))\leq 1/3.$ In which case the right hand side of \eqref{eq:53} can be bounded above by $\frac{1}{3}\frac{\log 2}{-\log \alpha}\approx 0.281914.$ However by (\ref{eq:52}) this quantity is strictly less than our calculated dimension $\frac{\log \lambda}{-\log \alpha}\approx 0.644297.$

\end{example}

\begin{example}
Let $\alpha= \sqrt{2}-1 $ and $t=\frac{1}{\alpha(\alpha^3-1)}+\frac{1}{\alpha^2(1-\alpha^3)}.$ Then a simple calculation demonstrates that the set of $\alpha$-expansions of $t$ is precisely the set $\{ 0(-1)(-1),(-1)10\}^{\mathbb{N}}.$ Applying \eqref{eq:12} we see that
$$\G\cap (\G+t)=\Big\{\sum_{i=1}^{\infty}\epsilon_i\alpha^i:(\epsilon_i)\in \{100,000,010,011\}^{\mathbb{N}}\Big\}$$ This last set is clearly a self-similar set generated by four contracting similitudes of the order $\alpha^3$. This self-similar set satisfies the strong separation condition. So
\[\dim_{H}(\G\cap (\G+t))=\frac{\log 4}{-3\log \alpha}.\]
 Each $\alpha$-expansion of $t$ satisfies $d((t_i))=1/3$. Thus the right hand side of \eqref{eq:53} can be bounded above by $\frac{\log 2}{-3\log \alpha}.$ Thus this choice of $\alpha$ and $t$ gives another example where we have strict inequality within \eqref{eq:53} .

\end{example}

We now give an example of an $\alpha\in(1/3,1/2)$ and $t\in \G-\G$ for which $\G\cap (\G+t)$ contains only transcendental numbers. For $\alpha\in (0,1/3]$ examples are easier to construct, however, when $\alpha\in(1/3,1/2)$ the problem of multiple codings arises and a more delicate approach is required. Our examples arise from our proof of Theorem \ref{th:12} and make use of ideas from the well known construction of Liouville.

We call a number $x\in\mathbb{R}$ a Liouville number if for every $\delta>0$ the inequality $$|x-p/q|\leq q^{-(2+\delta)}$$ has infinitely many solutions. An important result states that every Liouville number is a transcendental number \cite{Bugeaud_2004}. This result will be critical in what follows.

\begin{example}
Let $p/q\in(1/3,\frac{3-\sqrt{5}}{2})$. Then there exists $t\in \u_{p/q}$ such that $\G\cap(\G+t)$ only contains Liouville numbers. For any sequences of integers $(n_k)_{k=1}^{\infty}$ the sequence $$(t_i)=(1(-1))^{n_{1}}\,0\,(1(-1))^{n_{2}}\,0\cdots$$
 is contained in $\U_{p/q}.$ Now let $(n_k)$ be a rapidly increasing sequence of integers such that
\begin{equation}
\label{eq:54}
\Big(\frac{q}{p}\Big)^{2n_{1}+\cdots+2n_{k+1}+k+1}\geq q^{k(2n_1+\cdots 2n_k+k+3)}
\end{equation} Let $x\in \Gamma_{p/q}\cap (\Gamma_{p/q}+t),$ then $x=\sum_{i=1}^{\infty}\epsilon_i\Big(\frac{p}{q}\Big)^{i}$ where $\epsilon_i=1$ if $t_{i}=1,$ $\epsilon_i=0$ if $t_i=-1,$ and $\epsilon_i\in\set{0,1}$ if $t_i=0.$ It follows from our choice of $(t_i)$ that $$(\epsilon_i)=(10)^{n_{1}}\epsilon_{2n_{1}+1}(10)^{n_{2}}\epsilon_{2n_{1}+2n_{2}+2}\cdots.$$

For each $k\in\mathbb{N}$ we consider the rational
\begin{equation}
\label{eq:55}
\frac{p_k}{q_k}:=\sum_{i=1}^{2n_{1}+\cdots +2n_{k}+k}\epsilon_i\Big(\frac{p}{q}\Big)^{i}+\Big(\frac{p}{q}\Big)^{2n_{1}+\cdots +2n_{k}+k+1}\sum_{i=0}^{\infty}(p/q)^{2i},
\end{equation}
where $p_k$ and $q_k$ are coprime.
 Either the block $00$ or $11$ occurs infinitely often within $(\epsilon_i)$. So $p_{k}/q_{k}\neq x$. Importantly, if we expand the right hand side of (\ref{eq:55}) we can bound the denominator by
\begin{equation}
\label{denominator bound}
q_{k}\leq q^{2n_{1}+\cdots +2n_k+k+3}.
\end{equation}The $p/q$-expansion on $p_{k}/q_{k}$ agrees with that of $x$ upto the first $(2n_{1}+\cdots +2n_{k+1}+k)$ position. Therefore
\begin{equation}
\label{eq:56}
|x-p_k/q_k|\leq c\cdot  \Big(\frac{p}{q}\Big)^{2n_{1}+\cdots +2n_{k+1}+k+1}
\end{equation} for some constant $c$. Combining \eqref{eq:54}, \eqref{denominator bound}, and \eqref{eq:56} we see that for each $k\in \mathbb{N}$ $$|x-p_k/q_k|\leq c q_{k}^{-k}.$$ Therefore $x$ is a Liouville number. Since $x$ was arbitrary, every $x\in \Gamma_{p/q}\cap(\Gamma_{p/q}+t)$ is Liouville.

\end{example}

\section*{Acknowledgements}
The authors are grateful to Wenxia Li for being a good source of discussion and for his generous hospitality.


\end{document}